\RequirePackage{fix-cm}
\documentclass[smallextended]{svjour3}       
\usepackage{algorithmic}

\usepackage{array}
\usepackage{amsmath,epsfig,comment}
\numberwithin{equation}{section}
\usepackage{tcolorbox}
\usepackage{amssymb}
\smartqed  
\usepackage{graphicx}
\usepackage{amstext}
\usepackage{mathrsfs}
\usepackage{multirow}
\usepackage{threeparttable}
\usepackage{subfigure}
\usepackage[multiple]{footmisc}
\usepackage{color}

\usepackage{amssymb,amsmath,cite}
\usepackage{epsfig}
\usepackage{color}
\usepackage{bm}
\usepackage[ruled,vlined,linesnumbered]{algorithm2e}
\usepackage{graphicx}
\usepackage{subfigure}
\usepackage{algorithmic}
\usepackage{multirow}
\usepackage{cite}
\usepackage{hyperref}
\usepackage{tikz}
\usetikzlibrary{arrows.meta,positioning}
\hypersetup{
	colorlinks=true,
	linkcolor=black,
	filecolor=magenta,
	urlcolor=cyan,
	citecolor=blue
}
\usepackage{extarrows}

\makeatletter
\def\@thmcountersep{.}
\makeatother
\newtheorem{assumption}{Assumption}[section]{\bfseries}{\normalfont}
\spnewtheorem{prop}{Proposition}[section]{\bfseries}{\normalfont}
\spnewtheorem{lem}{Lemma}[section]{\bfseries}{\normalfont}
\spnewtheorem{theo}{Theorem}[section]{\bfseries}{\normalfont}
\DeclareMathOperator*{\argmin}{arg\,min}

\usepackage{float}
\usepackage{booktabs}
\usepackage{multirow}

\DeclareMathAlphabet\mathbfcal{OMS}{cmsy}{b}{n}

\newcommand{\rgrad}{\mathrm{grad}\,}

\newcommand{\prox}{\mathrm{prox}}
\newcommand{\proj}{\mathrm{proj}}
\begin{document}
%
	\title{On the Oracle Complexity of a Riemannian Inexact Augmented Lagrangian Method for Riemannian Nonsmooth Composite Problems
	}
	
	\titlerunning{Oracle Complexity of an RiAL Method for Riemannian Nonsmooth Problems}        
	
	\author{Meng Xu       \and
Bo Jiang \and
Ya-Feng Liu \and Anthony Man-Cho So
}


\institute{Meng Xu\at
LSEC, ICMSEC, Academy of Mathematics and Systems Science, Chinese Academy of Sciences, and University of Chinese Academy of Sciences, Beijing, China. \\
\email{xumeng22@mails.ucas.ac.cn}           
\and Bo Jiang \at
Ministry of Education Key Laboratory of NSLSCS, School of Mathematical Sciences, Nanjing Normal University, Nanjing, China.\\ \email{jiangbo@njnu.edu.cn}
\and Ya-Feng Liu\at
LSEC, ICMSEC, Academy of Mathematics and Systems Science, Chinese Academy of Sciences, Beijing, China. \\ \email{yafliu@lsec.cc.ac.cn}
\and Anthony Man-Cho So \at Department of Systems Engineering and Engineering Managment, The Chinese University of Hong Kong, HKSAR, China.\\
\email{ manchoso@se.cuhk.edu.hk}\\
}
\date{Received: date / Accepted: date}

\maketitle
\begin{abstract}
In this paper, we establish for the first time the oracle complexity of a Riemannian inexact augmented Lagrangian (RiAL) method with the classical dual update for solving a class of Riemannian nonsmooth composite problems. By using the Riemannian gradient descent method with a specified stopping criterion for solving the inner subproblem, we show that the RiAL method can find an $\varepsilon$-stationary point of the considered problem with $\mathcal{O}(\varepsilon^{-3})$ calls to the first-order oracle. This achieves the best oracle complexity known to date. Numerical results demonstrate that the use of the classical dual stepsize is crucial to the high efficiency of the RiAL method. 
\end{abstract}


%

\section{Introduction}\label{sec:introduction}
In this paper, we consider the following Riemannian nonsmooth composite problem:
\begin{equation}\label{prob:Riemannian nonsmooth composite}
	\min_{x\in\mathcal{M}}  \, \left\{ \Phi(x):=f(x)+h(\mathcal{A}(x))\right\},
\end{equation}
where $\mathcal{M}$ is a Riemannian submanifold embedded in a finite-dimensional Euclidean space $\mathcal{E}_1$, 
 $f:\mathcal{E}_1\to\mathbb{R}$ is a continuously differentiable function, $\mathcal{A}:\mathcal{E}_1\to\mathcal{E}_2$ is a smooth mapping with $\mathcal{E}_2$ being another finite-dimensional Euclidean space, and $h:\mathcal{E}_2\to \mathbb{R}$ is a convex Lipschitz continuous function with a tractable proximal mapping. Many problems in machine learning and signal processing can be formulated as problem \eqref{prob:Riemannian nonsmooth composite}, such as sparse principal component analysis (PCA) \cite{zou2018selective,jolliffe2003modified}, fair PCA \cite{samadi2018price,zalcberg2021fair,xu2023efficient2}, sparse canonical correlation analysis (CCA) \cite{deng2024oracle,chen2019alternating,hardoon2011sparse}, and beamforming design \cite{liu2024survey}. A variety of algorithms can be applied to tackle problem \eqref{prob:Riemannian nonsmooth composite}, including  Riemannian subgradient-type methods \cite{borckmans2014riemannian,hosseini2017riemannian,hosseini2018line,li2021weakly,hu2023constraint}, Riemannian proximal gradient-type methods \cite{chen2020proximal,huang2022riemannian,huang2023inexact,wang2022manifold,chen2024nonsmooth}, Riemannian smoothing-type algorithms \cite{beck2023dynamic,peng2023riemannian,zhang2023riemannian}, splitting-type methods \cite{lai2014splitting,kovnatsky2016madmm,deng2023manifold,li2023riemannian,zhou2023semismooth,deng2024oracle}, and Riemannian min-max algorithms \cite{xu2023efficient2,xu2024riemannian}. Among the previously mentioned algorithms, the Riemannian augmented Lagrangian (AL) method has demonstrated advantages in addressing the general mapping $\mathcal{A}$ along with possible additional constraints in problem \eqref{prob:Riemannian nonsmooth composite} \cite{zhou2023semismooth}. In this paper, we focus on the Riemannian AL method for solving problem \eqref{prob:Riemannian nonsmooth composite}.

As a powerful algorithmic framework for constrained problems, the AL method has been 
extensively studied since 1960s \cite{hestenes1969multiplier,powell1969method}.  At each iteration, it updates the primal variable by (approximately) minimizing the AL function followed by a dual (gradient ascent) step to update the dual variable. 
Recently, the AL method has been generalized to tackle optimization problems with Riemannian manifold constraints (e.g., problem \eqref{prob:Riemannian nonsmooth composite}), resulting in various efficient Riemannian AL methods. Below, we briefly introduce several such algorithms, which maintain the manifold constraint within the subproblem when solving problem \eqref{prob:Riemannian nonsmooth composite}. For some recent advances in AL methods for solving general nonsmooth nonconvex optimization problems, one can refer to \cite{li2021rate,sahin2019inexact,kong2023iteration,dahal2023damped} and the references therein. 

When $\mathcal{A}$ is a linear mapping, 
Deng and Peng \cite{deng2023manifold}, and Deng et al. \cite{deng2024oracle} proposed two types of Riemannian inexact AL methods for solving problem \eqref{prob:Riemannian nonsmooth composite} with a compact manifold $\mathcal{M}$. The first one has asymptotic convergence, while the second achieves the best-known first-order oracle complexity of $\mathcal{O}(\varepsilon^{-3})$ to attain an $\varepsilon$-stationary point. To establish the convergence or complexity results, additional requirements are imposed on the dual updates therein. Specifically, in \cite{deng2023manifold}, the dual update is followed by a projection onto a specified compact set at each iteration, in order to ensure the boundedness of the (Lagrange) multiplier sequence. In contrast, the work \cite{deng2024oracle} employed a damped technique to satisfy the same boundedness requirement for the multiplier sequence.  
However, as observed in \cite{kong2023iteration}, such damped dual stepsizes may slow down the convergence of AL-like methods.  
When $\mathcal{A}$ is a general nonlinear mapping, Zhou et al. \cite{zhou2023semismooth} proposed a manifold-based AL (MAL) method with classical dual updates and established its global convergence.  However, the oracle complexity of their approach remains unclear.
In summary, the oracle complexity of  the Riemannian AL method with the classical dual update for solving problem \eqref{prob:Riemannian nonsmooth composite} involving a general nonlinear mapping $\mathcal{A}$ remains unknown. 
Given this background, \textit{we are motivated to investigate the oracle complexity of such method
for solving problem \eqref{prob:Riemannian nonsmooth composite}.} 

\begin{table}[t]
	\centering
	\renewcommand{\arraystretch}{1.2}
	\tabcolsep=0.15cm
	\caption{Comparison of state-of-the-art Riemannian AL methods for solving problem \eqref{prob:Riemannian nonsmooth composite}.} 
		\begin{tabular}{c|c|c|c}
			\hline
			{Algorithm} & {Mapping $\mathcal{A}$}  & Dual Stepsize & Complexity\\ 
			\hline
			MIAL \cite{deng2023manifold}&linear &classical but with projection& ---\\ 
			\hline
			MAL \cite{zhou2023semismooth} & nonlinear  &classical& --- \\
			\hline
			ManIAL \cite{deng2024oracle}&linear & damped & $\mathcal{O}(\varepsilon^{-3})$ \\
			\hline
			RiAL~(this paper) & nonlinear  & classical & $\mathcal{O}(\varepsilon^{-3})$ \\
			\hline \end{tabular}
		\label{tab: complexity comparision}
	\end{table}

In this paper, we propose a Riemannian inexact AL (RiAL) method, where each subproblem is solved to a specified accuracy using the Riemannian gradient descent (RGD) method. We establish its first-order oracle complexity of $\mathcal{O}(\varepsilon^{-3})$ for finding an $\varepsilon$-stationary point of problem \eqref{prob:Riemannian nonsmooth composite}. 
As seen from Table \ref{tab: complexity comparision}, which summarizes the applicability and complexity of different Riemannian AL methods for solving problem \eqref{prob:Riemannian nonsmooth composite}, our proposed approach is able to tackle more general problem settings and achieves the best-known oracle complexity result.
Additionally, we present numerical results on sparse PCA and sparse CCA to demonstrate that the proposed RiAL method outperforms the ManIAL method in terms of computational efficiency. This suggests that the damped dual stepsize used in ManIAL slows down the convergence of the algorithm.

\section{Notation and Preliminaries}\label{sec:preliminaries}
We begin by introducing the notation and some concepts in Riemannian optimization \cite{absil2008optimization,boumal2023introduction}. 
Let $\langle \,\!\cdot\,, \cdot\rangle$ and $\| \cdot \| $ denote the standard inner product and its induced norm on the Euclidean space $\mathcal{E}$, respectively. 
Let $\mathcal{M}$ be a Riemannian submanifold embedded in $\mathcal{E}$ and $ \mathrm{T}_{x}\mathcal{M} $ denote the tangent space to $\mathcal{M}$ at $x\in\mathcal{M}$.
Throughout this paper, the Riemannian metric on $\mathcal{M}$ is induced from the standard Euclidean product. 
The Riemannian gradient of a smooth function $f: \mathcal{E} \rightarrow \mathbb{R}$ at a point $x \in \mathcal{M}$ is given by
$\rgrad f(x)=\proj_{\mathrm{T}_{x}\mathcal{M}}(\nabla f(x))$, where $\nabla f(x)$ is the Euclidean gradient of $f$ at $x$ and $\proj_{\mathrm{T}_x \mathcal{M}}(\cdot)$ is the Euclidean projection operator onto $\mathrm{T}_x \mathcal{M}$. 
A retraction at $x \in \mathcal{M}$ is a smooth mapping $\mathrm{R}_x: \mathrm{T}_x \mathcal{M} \to \mathcal{M}$ satisfying (i) $\mathrm{R}_x(\mathbf{0}_x) = x$, where $\mathbf{0}_x$ is the zero element in $\mathrm{T}_x \mathcal{M}$; (ii) $\frac{\mathrm{d}}{\mathrm{d} t} \mathrm{R}_x (t v)|_{t = 0} = v$ for all $v \in \mathrm{T}_x \mathcal{M}$. Without loss of generality, we assume that the retraction $\mathrm{R}_x$ is globally defined over $\mathrm{T}_x \mathcal{M}$.

Next, we introduce some necessary notions in convex analysis \cite{rockafellar2009variational,beck2017first}. 
For a subset $\mathcal{X}$ in $\mathcal{E}$, we use $\mathrm{conv}\,\mathcal{X}$ to denote the convex hull of $\mathcal{X}$.
Let $h:\mathcal{E}\to \mathbb{R}$ be a convex function. 
For a given constant $\lambda > 0$, the proximal mapping and the Moreau envelope of $h$ are defined as 
\begin{equation*}
	\prox_{ \lambda h}(x)=\argmin_{u\in\mathcal{E}}\,\left\{h(u)+\frac{1}{2\lambda}\|u-x\|^2\right\},
\end{equation*} 
and
\begin{equation*}
	M_{{\lambda h}}(x)=\min_{u\in\mathcal{E}}\,\left\{h(u)+\frac{1}{2\lambda}\|u-x\|^2\right\},
\end{equation*}
respectively.
The following theorem characterizes several useful properties related to the subgradient and the Moreau envelope.
\begin{theo}\label{gradient of Moreau}\textbf{(\hspace{-0.01cm}\cite[Theorems  3.61, 6.39, and 6.67]{beck2017first})}
	Let $h:\mathcal{E}\to\mathbb{R}$ be a  convex $L_h$-Lipschitz continuous function and $h^*$ be its conjugate function. Then, $\|g\|\leq L_h$ for any $g\in\partial h(x)$ and $ x\in\mathcal{E}$. Moreover, for a given constant $\lambda>0$ and any $x\in\mathcal{E}$, 
  \begin{equation*}
 M_{\lambda h}(x) = \max_{z \in \mathcal{E}}\left\{ \langle x, z\rangle -h^*(z) - \frac{\lambda}{2} \|z\|^2\right\}
 \end{equation*}
 and
	\begin{equation*}\label{gradient of Moreau envelope}
		\nabla M_{\lambda h}(x)=\frac{1}{\lambda}(x-\prox_{\lambda h}(x))\in\partial h(\prox_{\lambda h}(x)).
	\end{equation*}
\end{theo}

\section{The RiAL Method and Its Oracle Complexity}
In this section, we introduce the RiAL method for solving problem \eqref{prob:Riemannian nonsmooth composite} and establish its first-order oracle complexity.
\subsection{The RiAL Method}
The key challenge in solving problem \eqref{prob:Riemannian nonsmooth composite} stems from the presence of both a manifold constraint and a nonsmooth objective function. To tackle this, existing Riemannian AL methods split these two difficulties. 
Specifically, by introducing an auxiliary variable $y$, problem \eqref{prob:Riemannian nonsmooth composite} can be  equivalently reformulated as 
\begin{equation}\label{prob: eq prob for ALM}
	\min_{x\in\mathcal{M},\,y\in\mathcal{E}_2} f(x) +h(y) \quad \mathrm{s.t.} \quad \mathcal{A}(x) - y = 0.
\end{equation}
The augmented Lagrangian function associated with problem \eqref{prob: eq prob for ALM} is defined as
\begin{equation*}\label{equ: augmented Lagrangian}
	\mathcal{L}_\sigma(x, y; z):= f(x)+h(y)+\langle z,\,\mathcal{A}(x)- y\rangle+\frac{\sigma}{2}\|\mathcal{A}(x)- y\|^2,
\end{equation*}
where $z$ is the Lagrange multiplier (also called dual variable) corresponding to the constraint $\mathcal{A}(x) - y = 0$ and $\sigma > 0$ is the penalty parameter. 
At the $k$-th iteration, an ordinary Riemannian AL method updates the next point as
\begin{equation*}\label{prob: RALM subprob}
	(x_{k+1},y_{k+1})\approx\argmin_{x \in \mathcal{M},\,y\in\mathcal{E}_2} \mathcal{L}_{\sigma_k}(x, y;z_k).
\end{equation*}
Observe that for any fixed $x\in \mathcal{M}$, the optimal $y$ in the above minimization problem can be expressed as $y^* = \mathrm{prox}_{h/\sigma_k}(\mathcal{A}(x) + z_k/\sigma_k)$. Similar to the Riemannian AL methods in \cite{deng2023manifold,deng2024oracle,zhou2023semismooth},  we update $x_{k+1}$ and $y_{k+1}$ via the following scheme: 
\begin{subequations}
  \begin{align}%
x_{k+1}&\approx\argmin_{x\in\mathcal{M}}\,\mathcal{L}_k(x),\label{RALM subproblem}\\
	y_{k+1}&=\prox_{h/\sigma_k}\left(\mathcal{A}(x_{k+1})+\frac{z_k}{\sigma_k}\right).\label{equ: update y}
	\end{align}  
\end{subequations}
Here, 
\begin{equation}\label{equ: L_k}
\mathcal{L}_k(x):=\min_{y\in\mathcal{E}_2}\,\mathcal{L}_{\sigma_{k}}(x,y;z_k)=f(x)+M_{{h/\sigma_k}}\left(\mathcal{A}(x)+\frac{{z_k}}{\sigma_{k}}\right).
\end{equation}
Let $\nabla \mathcal{A}^\top (x)$ be the adjoint mapping of $\nabla \mathcal{A}(x)$. We know from Theorem \ref{gradient of Moreau} that $\mathcal{L}_k$ is differentiable and 
\begin{equation}\label{equ: gradient of L_k}
	\begin{aligned}
		\nabla \mathcal{L}_k(x) & = \nabla f(x) + \sigma_k \nabla\mathcal{A}(x)^\top\left(\mathcal{A}(x)+\frac{z_k}{\sigma_k} - \mbox{prox}_{ h/\sigma_k}\left(\mathcal{A}(x)+\frac{z_k}{\sigma_k} \right)\right).
	\end{aligned}
\end{equation}
Therefore, we propose to use the RGD method to solve subproblem \eqref{RALM subproblem}. Specifically, starting from an initial point $x_{k,1} = x_{k}$, the iteration of RGD for $ t\geq1 $ is given by
\begin{equation}\label{equ:RGD}
	x_{k,t+1}=\mathrm{R}_{x_{k,t}}(-\zeta_{k,t}\rgrad\mathcal{L}_k(x_{k,t})),
\end{equation}
where $\zeta_{k,t}>0$ is the stepsize determined later in 
Theorem \ref{Th: I}.
The RiAL method with RGD (RiAL-RGD) for solving problem \eqref{prob:Riemannian nonsmooth composite} is formally presented in Algorithm \ref{Algorithm RiAL}.
\begin{algorithm}[t]\label{Algorithm RiAL}
	\fontsize{10pt}{\baselineskip}\selectfont
	\caption{RiAL-RGD for solving problem \eqref{prob:Riemannian nonsmooth composite}}
	Input $x_1\in\mathcal{M}$, $y_1=z_1=\mathbf{0}$, $\varepsilon_1, \sigma_1>0$, $b>1$. 
	
	\For{$k=1,\,2,\,\ldots$}
	{Apply the RGD method \eqref{equ:RGD} with $x_{k,1} = x_{k}$ for $t = 1, 2, \ldots$ until 
	\begin{equation}\label{ineq: subproblem stop}
		\|\rgrad \mathcal{L}_k(x_{k, t_k})\|\leq \varepsilon_k
	\end{equation}
		holds for some positive integer $t_k$ and set $x_{k+1} = x_{k, {t_k}}$.   
		
		Update the auxiliary variable via \eqref{equ: update y}.
		
		Update the dual variable:
		\begin{equation}\label{equ: update z}
			z_{k+1}=z_k+\sigma_k(\mathcal{A}(x_{k+1})-y_{k+1}).
		\end{equation}
		
		Set
		\begin{equation}\label{equ: sigma_k and epsilon_k}
\sigma_{k+1}= b \sigma_k  \quad\text{and}\quad\varepsilon_{k+1}=\varepsilon_k/b.
		\end{equation}
		
	}
\end{algorithm}

\subsection{Oracle Complexity}
In this subsection, we establish the first-order oracle complexity of RiAL-RGD. Before presenting our results, we first introduce the definitions of the first-order oracle and a commonly used notion of $\varepsilon$-stationary point of problem \eqref{prob:Riemannian nonsmooth composite} (see, e.g., \cite{deng2024oracle,li2023riemannian,xu2024riemannian}).
\begin{definition}
	For problem \eqref{prob:Riemannian nonsmooth composite}, given $x\in\mathcal{M}$ and $y\in\mathcal{E}_2$, the first-order oracle returns  $f(x)$,  $\nabla f(x)$,  $\mathcal{A}(x)$, $\nabla \mathcal{A}(x)$, and $\prox_{ \lambda h}(y)$ for any $\lambda>0$.
\end{definition}
\begin{definition}\label{def:epsilon-deter}
	For any given $\varepsilon>0$, we say that $x\in\mathcal{M}$ is an $\varepsilon$-stationary point of problem \eqref{prob:Riemannian nonsmooth composite} if there exist $y \in \mathcal{E}_2$ and $z \in \partial h(y)$ such that 
	\begin{equation*}\label{eq:epsi-kkt}
 \max\left\{\left\|\proj_{\mathrm{T}_x\mathcal{M}}\left(\nabla f(x) +\nabla \mathcal{A}(x)^\top z \right) \right\|,\,\| \mathcal{A}(x) - y\|\right\}\leq\varepsilon.
	\end{equation*}
\end{definition}

We make the following assumptions for our analysis (see, e.g., \cite{boumal2019global,chen2020proximal,xu2024riemannian,zhou2023semismooth}).
\begin{assumption}\label{assumption: compact level set}
	The level set 
	\begin{equation*}\label{Omega}
		\Omega_{x_1} := \{x \in \mathcal{M} \mid \Phi(x) \leq \Phi(x_1)+\Upsilon\}
	\end{equation*}
	is compact, where $\Phi$ is defined in \eqref{prob:Riemannian nonsmooth composite}, $x_1$ is the initial point of Algorithm \ref{Algorithm RiAL}, and
	\begin{equation}\label{equ:Upsilon}
		\Upsilon:=\sum_{k = 1}^{+\infty}\frac{1}{\sigma_{k}}=\frac{b}{\sigma_1(b-1)}.
	\end{equation}
\end{assumption}
\begin{assumption}\label{assumption: bounded below}
	The function $\Phi$ defined in \eqref{prob:Riemannian nonsmooth composite} is bounded from below on $\mathcal{M}$, namely, $\Phi^*:=\inf_{x\in\mathcal{M}}\Phi(x)>-\infty$.
\end{assumption}
\begin{assumption}\label{assumption1}
	The functions $f$ and $h$ and the smooth mapping $\mathcal{A}$ satisfy the following conditions:
	\begin{itemize}
		\item[(i)] The function $f: \mathcal{E}_1 \to \mathbb{R}$ is continuously differentiable and satisfies the descent property over $\mathcal{M}$ and $h$ is $L_h$-continuous, i.e., 
		\begin{align}
			f(x') \leq f(x) + \langle \nabla f(x),\,x' - x\rangle + \frac{L_f}{2} \|x' - x\|^2, \quad \forall\,x,x' \in \mathcal{M},\\
			\|h(x)-h(x')\|\leq L_h\|x-x'\|,\quad\forall\,x,x'\in\mathcal{E}_1.\label{L_h}
		\end{align}
		\item[(ii)] The mapping $\mathcal{A}: \mathcal{E}_1 \to \mathcal{E}_2$ and its Jacobian mapping $\nabla \mathcal{A}$ are $L_{\mathcal{A}}^0$-Lipschitz and $L_{\mathcal{A}}^1$-Lipschitz continuous over $\mathrm{conv}\,\mathcal{M}$, respectively. 
		In other words, for any $x,x' \in \mathrm{conv}\,\mathcal{M}$, there hold that
		\begin{subequations}
			\begin{align}
				\| \mathcal{A}(x) - \mathcal{A}(x') \| \leq L_{\mathcal{A}}^0 \|x - x'\|,\label{L_A^0}\\
				\| \nabla \mathcal{A}(x) - \nabla \mathcal{A}(x') \| \leq L_{\mathcal{A}}^1 \|x - x'\|\label{L_A^1}. 
			\end{align}
		\end{subequations}
		 \item[(iii)] The Jacobian mapping $\nabla\mathcal{A}$ is bounded over $\mathrm{conv}\,\mathcal{M}$, i.e.,
		 \begin{equation}\label{rhoA}
		 	{\rho_\mathcal{A}}:=\max_{x \in\mathrm{conv}\mathcal{M}} \|\nabla \mathcal{A}(x)\| <+\infty.
		 \end{equation}
	\end{itemize}
\end{assumption}
The following lemma, which is extracted from \cite[Appendix B]{boumal2019global}, shows that the retraction satisfies the first- and second-order boundedness conditions.
\begin{lem}\label{lemma Bound retraction}
	Suppose that Assumption \ref{assumption: compact level set} holds. Then, there exist constants  $\alpha_1, \alpha_2>0$ such that 
	\begin{equation*}\label{bound retraction}
		\|\mathrm{R}_x(v)-x\|\leq\alpha_1\|v\| \quad \text{{and}} \quad \|\mathrm{R}_x(v)-x-v\|\leq\alpha_2\|v\|^2
	\end{equation*} 
	for any $x\in\Omega_{x_1}$ and $v\in\mathrm{T}_{x}\mathcal{M}$.
\end{lem}

We next establish the descent property of $\mathcal{L}_k$ in \eqref{equ: L_k} as follows.
\begin{lem}\label{lemma l-smooth }
	Suppose that Assumptions \ref{assumption: compact level set} and \ref{assumption1} hold. Then, the function $\mathcal{L}_k$ in \eqref{equ: L_k} satisfies the following properties: 
	\begin{itemize}
		\item[(i)] {\bf Euclidean Descent.} For any $x,x' \in \mathcal{M}$, we have
		\begin{equation}\label{lemma l-smooth E ineq E}
			{\mathcal{L}_k}(x') \leq {\mathcal{L}_k}(x) + \langle \nabla {\mathcal{L}_k}(x),\,x' - x \rangle + \frac{\ell_k}{2} \|x' - x\|^2, 
		\end{equation}
		where $\ell_k= L_f+L_hL_{\mathcal{A}}^1+\sigma_{k}\rho_\mathcal{A}L_\mathcal{A}^0.$
		\vspace{5pt}
		\item[(ii)] {\bf Riemannian Descent.} For any $x\in\Omega_{x_1}$ and $v\in\mathrm{T}_{x}\mathcal{M}$, 
		we have
		\begin{equation}\label{lemma l-smooth E ineq R}
			{\mathcal{L}_k}(\mathrm{R}_x(v))\leq {\mathcal{L}_k}(x)+\langle \rgrad\, {\mathcal{L}_k}(x),\,v\rangle + \frac{L_k(x)}{2}\|v\|^2,
		\end{equation} 
		where 	
		\begin{equation} \label{equ:Lk(x)} 
			L_k(x) = \ell_k\alpha_1^2  + 2\left(\|\nabla f(x)\|+\rho_\mathcal{A} L_h\right)\alpha_2.
		\end{equation}
	\end{itemize}
\end{lem}
\begin{proof}
	The proof is similar to that in \cite[Lemma 4.2]{xu2024riemannian}. We include the proof here for completeness.
	For simplicity of notation, denote 
	\begin{align}
		\mathcal{B}_k(x)&:=\sigma_{k}\mathcal{A}(x)+z_k-\sigma_{k}\prox_{h/\sigma_k}\left(\mathcal{A}(x)+\frac{z_k}{\sigma_{k}}\right),\nonumber\\
		\psi_k(x)&:= M_{h/\sigma_k} \left(\mathcal{A}(x) + \frac{z_k}{\sigma_{k}}\right). \label{equ: psi_k}
	\end{align}
	From Theorem \ref{gradient of Moreau}, we know that
	\begin{equation*}
		\nabla \psi_k(x) =  \nabla\mathcal{A}(x)^\top \mathcal{B}(x) \quad\text{and}\quad\mathcal{B}_k(x)\in\partial h\left(\prox_{h/\sigma_k}\left(\mathcal{A}(x)+\frac{z_k}{\sigma_{k}}\right)\right).
	\end{equation*}
	Since $h$ is $L_h$-Lipschitz continuous as assumed in \eqref{L_h}, from Theorem \ref{gradient of Moreau}, we have
	\begin{equation}\label{inequ: B <=Lh}
		\|\mathcal{B}_k(x)\|\leq L_h,\quad \forall\, x\in\mathcal{M},\,k\geq1.
	\end{equation}
	
	We now show that $\nabla\psi_k$ is Lipschitz continuous. Specifically, for any $x,x'\in\mathrm{conv}\,\mathcal{M}$, it holds that
	\begin{align*}
		&\left\|\nabla \psi_k(x)-\nabla \psi_k(x')\right\|\\
		\leq{}&\| (\nabla\mathcal{A}(x)-\nabla\mathcal{A}(x'))^\top \mathcal{B}_k(x)\|+\| \nabla\mathcal{A}(x')^\top (\mathcal{B}_k(x)-\mathcal{B}_k(x'))\|\\
		\overset{(\text{a})}{\leq}{}& L_\mathcal{A}^1\|\mathcal{B}_k(x)\|\cdot\|x-x'\|+\rho_\mathcal{A}\|\mathcal{B}_k(x)-\mathcal{B}_k(x')\|\\
		\overset{(\text{b})}{\leq}{}& L_\mathcal{A}^1\|\mathcal{B}_k(x)\|\cdot\|x-x'\|+\rho_\mathcal{A}\sigma_{k}\|\mathcal{A}(x)-\mathcal{A}(x')\|\\
		\overset{(\text{c})}{\leq}{}& \left(L_hL_\mathcal{A}^1+\sigma_{k}\rho_\mathcal{A}L_\mathcal{A}^0\right)\|x-x'\|,
	\end{align*}
	where (a) is due to \eqref{L_A^1} and \eqref{rhoA}, (b) is due to the firm nonexpansiveness of the proximal operator (see \cite[Theorem 6.42]{beck2017first}), and (c) is due to \eqref{L_A^0} and \eqref{inequ: B <=Lh}.
	Hence,  for any $x, x' \in \mathrm{conv} \mathcal{M}$, it follows from \cite[Lemma 1.2.3]{nesterov2018lectures} that
	\[
	\psi_k(x') \leq \psi_k(x)+\langle\nabla\psi_k(x), x'-x\rangle+\frac{L_hL_\mathcal{A}^1+\sigma_{k}\rho_\mathcal{A}L_\mathcal{A}^0}{2}\|x-x'\|^2.
	\]
	This, together with the definition of $\mathcal{L}_k$ in \eqref{equ: L_k} and the definition of $\psi_k$ in \eqref{equ: psi_k}, 
	implies the desired Euclidean descent property in \eqref{lemma l-smooth E ineq E} over $\mathcal{M}$.  
	
	Moreover, following the similar analysis in \cite[Appendix B]{boumal2019global}, we know that $\mathcal{L}_k$ in \eqref{equ: L_k} also satisfies the Riemannian descent property in \eqref{lemma l-smooth E ineq R}.
\end{proof}

Next, we present some important inequalities related to $\mathcal{L}_k$.
\begin{lem}
	Suppose that Assumption \ref{assumption1} holds. Then, for any $x\in\mathcal{M}$ and $k\geq1$, $\mathcal{L}_k$ defined in \eqref{equ: L_k} satisfies
	\begin{align}
		\mathcal{L}_k(x)&\geq\Phi^*{-\frac{2L_h^2}{\sigma_1}},\label{inequ: Lk>Phi*}\\
		\mathcal{L}_k(x)&\leq\Phi(x)+\frac{L_h^2}{\sigma_{k}},\label{inequ: L_k<Phi}\\
		\mathcal{L}_{k+1}(x)&\leq \mathcal{L}_k(x)+\frac{2L_h^2}{\sigma_{k}}\label{inequ: Lk+1<Lk+c}.
	\end{align}
\end{lem}
\begin{proof}
	By the optimality of $y_{k+1}$ in \eqref{equ: update y}, we have
	\begin{equation*}
		z_k +  \sigma_k(\mathcal{A}(x_{k+1}) - y_{k+1}) \in  \partial h(y_{k+1}),
	\end{equation*}
	which, together with \eqref{equ: update z}, implies  
	\begin{equation}\label{eq:y-epsilonk}
		z_{k+1} \in \partial h(y_{k+1}).
	\end{equation}
	Since $h$ is $L_h$-Lipschitz continuous as assumed in \eqref{L_h}, by using Theorem \ref{gradient of Moreau} and noting that $z_1 = \mathbf{0}$, we have
	\begin{equation}\label{equ:new}
		\| z_{k}\|  \leq L_h, \quad \forall\,k \geq 1. 
	\end{equation}
By \eqref{equ:new}, the $L_h$-Lipschitz continuity of $h$, the definition of $\Phi$ in \eqref{prob:Riemannian nonsmooth composite}, and \cite[Theorem 10.51]{beck2017first}, for any $x \in \mathcal{M}$, we have the following inequalities:  
		\begin{equation}\label{ineq: M<f<M+:2}
				\mathcal{L}_k(x) \geq f(x) + h\left(\mathcal{A}(x)+ \frac{z_k}{\sigma_{k}}\right)  - \frac{L_h^2}{2 \sigma_k} \geq \Phi(x) - \frac{3L_h^2}{2\sigma_k}
			\end{equation}
   and 
   \begin{equation}\label{ineq: M<f<M+}
				\mathcal{L}_k(x) \leq f(x) + h\left(\mathcal{A}(x)+ \frac{z_k}{\sigma_{k}}\right) \leq \Phi(x) + \frac{L_h^2}{\sigma_k}.
			\end{equation}
   With the update of  $\sigma_k$ in \eqref{equ: sigma_k and epsilon_k} and the definition of $\Phi^*$ in Assumption \ref{assumption: bounded below}, we immediately obtain \eqref{inequ: Lk>Phi*} and \eqref{inequ: L_k<Phi} from \eqref{ineq: M<f<M+:2} and \eqref{ineq: M<f<M+}, respectively. 

It remains to prove \eqref{inequ: Lk+1<Lk+c}. First, we show that, for any $\lambda_1, \lambda_2 > 0$ and $w,w'\in\mathcal{E}$, we have
	\begin{equation}
 \label{equ:M:lips:general}
 		\begin{aligned}
			&M_{\lambda_1h}(w)-M_{\lambda_2h}(w')\\
	\overset{\text{(a)}}{=}{}&\max_{z\in\mathcal{E}}\left\{\langle w,z\rangle-h^*(z)-\frac{\lambda_1}{2}\|z\|^2\right\}-\max_{z\in\mathcal{E}}\left\{\langle w',z\rangle-h^*(z)-\frac{\lambda_2}{2}\|z\|^2\right\}\\
	\overset{\text{(b)}}{\leq}{}&\langle w,z^*\rangle-h^*(z^*)-\frac{\lambda_1}{2}\|z^*\|^2-\langle w',z^*\rangle+h^*(z^*)+\frac{\lambda_2}{2}\|z^*\|^2\\
	={}&\langle w-w',z^*\rangle+\frac{\lambda_2-\lambda_1}{2}\|z^*\|^2 
	\overset{(c)}{\leq}  L_h\|w-w'\|+\frac{\lambda_2-\lambda_1}{2}L_h^2,
		\end{aligned}
	\end{equation}
	where (a) follows from Theorem \ref{gradient of Moreau}, (b) holds by the optimality of the maximization problems and $z^*:=\prox_{h^*/\lambda_1}(w/\lambda_1)$, and  (c) comes from the fact $\|z^*\|\leq L_h$, as shown in \cite[Theorem 4.23]{beck2017first}.
	Based on the two ends of \eqref{equ:M:lips:general}, 
 and noting $\sigma_{k+1} = b \sigma_k$ with $b> 1$ as in \eqref{equ: sigma_k and epsilon_k}, we have
	\[ 
		\begin{aligned}
			\mathcal{L}_{k+1}(x)-\mathcal{L}_k(x)={}&M_{h/\sigma_{k+1}}\left(\mathcal{A}(x)+\frac{z_{k+1}}{\sigma_{k+1}}\right)-M_{h/\sigma_{k}}\left(\mathcal{A}(x)+\frac{z_{k}}{\sigma_{k}}\right)\\
			\leq{}&\frac{1}{2}\left(\frac{1}{\sigma_{k}}-\frac{1}{\sigma_{k+1}}\right)L_h^2+L_h\left\|\frac{z_{k+1}}{\sigma_{k+1}}-\frac{z_{k}}{\sigma_{k}}\right\|\\
			\leq{}&\frac{1}{2}\left(\frac{1}{\sigma_{k}}-\frac{1}{\sigma_{k+1}}\right)L_h^2+\frac{L_h^2}{\sigma_{k+1}}+\frac{L_h^2}{\sigma_{k}}\leq\frac{2L_h^2}{\sigma_{k}},
		\end{aligned}
	\]
 which is the desired \eqref{inequ: Lk+1<Lk+c}. 
\end{proof}

Denote
\begin{equation}
\Delta :=\mathcal{L}_1(x_1)+2L_h^2\Upsilon-\Phi^*+\frac{2L_h^2}{\sigma_1}\label{equ: Delta}	
\end{equation}
and 
\[
	c_1:={\rho_\mathcal{A}L_\mathcal{A}^0\alpha_1^2},\quad  
	c_2:={(L_f+L_h L_\mathcal{A}^1)\alpha_1^2}+2\left(\max_{x \in \Omega_{x_1}}\|\nabla f(x)\|+\rho_\mathcal{A} L_h\right)\alpha_2.
\]
Recalling the definition of $L_k(x)$ in \eqref{equ:Lk(x)}, we derive a universal upper bound for $L_k(x_{k_,t})$ as   
\begin{equation}\label{equ:Lk}
L_k := \max_{x \in \Omega_{x_1}}L_k(x) = c_1 \sigma_k + c_2.  
\end{equation}

With the above results at hand, we start characterizing the inner iteration complexity of Algorithm \ref{Algorithm RiAL}.
\begin{lem}\label{lem: inner complexity}
	Suppose that Assumptions \ref{assumption: compact level set}, \ref{assumption: bounded below}, and \ref{assumption1} hold. Then, the inner RGD of Algorithm \ref{Algorithm RiAL} with $\zeta_{k,t}=1/L_k(x_{k,t})$ stops within at most $\left\lceil {2L_k\Delta}/{\varepsilon_k^{2}}\right\rceil$
	 iterations.
\end{lem}
\begin{proof}
	First, we show that if $x_{k,t}\in\Omega_{x_1}$, then
	\begin{equation}\label{inequ: bound rgrad}
		0\leq\frac{1}{2L_k(x_{k,t})}\|\rgrad\mathcal{L}_k(x_{k,t})\|^2\leq \mathcal{L}_k(x_{k,t})-\mathcal{L}_k(x_{k,t+1}).
	\end{equation}
	This follows directly by substituting $v=\rgrad\mathcal{L}_k(x_{k,t})/L_k(x_{k,t})$ into \eqref{lemma l-smooth E ineq R}. 
	Based on \eqref{inequ: L_k<Phi} and \eqref{inequ: Lk+1<Lk+c} and following the similar argument in \cite[Proposition 4.1]{xu2024riemannian}, we can further show that \eqref{inequ: bound rgrad} holds for any $k \geq 1$ and $t \geq 1$. Summing both sides of \eqref{inequ: bound rgrad} over $t = 1, 2, \dots, T_k$ gives  
	\begin{equation}
 \label{equ:sum:gradLk}
		\begin{aligned}
		\sum_{t=1}^{T_k}\frac{\|\rgrad\mathcal{L}(x_{k,t})\|^2}{2L_k(x_{k,t})}\leq \mathcal{L}_k(x_{k,1})-\mathcal{L}_k(x_{k,T_k+1}) 
  \overset{\text{(a)}}{\leq}\mathcal{L}_k(x_k)-\Phi^*+\frac{2L_h^2}{\sigma_1},
		\end{aligned}
	\end{equation}
        where (a) follows from \eqref{inequ: Lk>Phi*}.  
    From \eqref{equ:sum:gradLk}, we know that the inner RGD method must stop within a finite number of steps. With a slight abuse of notation, we assume that it stops at the $T_k$-th iteration. In other words, we have 
        \[
        \|\rgrad \mathcal{L}_k(x_{k,t})\| > \varepsilon_k,~t = 1, 2,\ldots, T_k - 1, \quad \|\rgrad \mathcal{L}_k(x_{k,T_k})\|\leq \varepsilon_k.
        \] 
        Using $L_k(x_k) \leq L_k$ as shown in \eqref{equ:Lk}, we can derive from \eqref{equ:sum:gradLk} that 
        \begin{equation}\label{equ:Tk:1st:estimate}
        T_k \leq \left\lceil \frac{2L_k(\mathcal{L}_k(x_k)-\Phi^*+{2L_h^2}/{\sigma_1})}{\varepsilon_k^2}\right\rceil. 
        \end{equation}

Next, we prove that $T_k\leq \lceil 2 L_k\Delta/\varepsilon_k^2 \rceil$. Since the inner RGD method stops at the $T_k$-th iteration, we have $x_{k+1} = x_{k, T_k}$. Using $x_{k} = x_{k,1}$ and \eqref{inequ: bound rgrad}, we find that $\mathcal{L}_{k}(x_{k+1}) \leq \mathcal{L}_{k}(x_k)$. This, together with \eqref{inequ: Lk+1<Lk+c}, implies that 
	\begin{equation*}
		\mathcal{L}_{k+1}(x_{k+1}) \leq \mathcal{L}_k(x_{k+1})+\frac{2L_h^2}{\sigma_{k}} \leq \mathcal{L}_k(x_{k})+\frac{2L_h^2}{\sigma_{k}},\quad\forall\,k\geq1.
	\end{equation*}
	Therefore, with \eqref{equ:Upsilon}, we obtain
	\begin{equation*}
		\mathcal{L}_{k+1}(x_{k+1})\leq\mathcal{L}_1(x_1)+2L_h^2\sum_{i=1}^{k}\frac{1}{\sigma_{i}}\leq\mathcal{L}_1(x_1)+2L_h^2\Upsilon,\quad\forall\,k\geq1.
	\end{equation*}
	Combining this with \eqref{equ: Delta} and \eqref{equ:Tk:1st:estimate} gives the desired result. 
\end{proof}

We are now ready to establish the first-order oracle complexity of Algorithm \ref{Algorithm RiAL} in finding an $\varepsilon$-stationary point of problem \eqref{prob:Riemannian nonsmooth composite}.
\begin{theo}\label{Th: I}
	Suppose that Assumptions \ref{assumption: compact level set}, \ref{assumption: bounded below}, and \ref{assumption1} hold. 
Then, for any given $\varepsilon>0$, Algorithm \ref{Algorithm RiAL} with ${\zeta_{k,t}=1/L_k(x_{k,t})}$ can find an $\varepsilon$-stationary point of problem \eqref{prob:Riemannian nonsmooth composite} with at most $\mathcal{O}(\varepsilon^{-3})$ first-order oracle calls.
\end{theo}
\begin{proof}
	First, we show that $x_{K+1}$ is an $\varepsilon$-stationary point of problem \eqref{prob:Riemannian nonsmooth composite} with 
	\begin{equation}\label{equ: K Option I}
			 	K := 1 + \left\lceil\log_{b}\left(\frac{\max\{2L_h\sigma_1^{-1},\,\varepsilon_1 \}}{ \varepsilon}\right)\right\rceil.
	\end{equation} 
	 By the updates of $z_{k+1}$ in \eqref{equ: update z} and $y_{k+1}$ in \eqref{equ: update y}, we have
	\begin{equation}\label{equ: z_k+1 proof}
		z_{k+1}=z_k+\sigma_k\left(\mathcal{A}(x_{k+1})-\prox_{h/\sigma_k}\left(\mathcal{A}(x_{k+1})+\frac{z_k}{\sigma_k}\right)\right).
	\end{equation}
	It follows from  \eqref{equ: gradient of L_k}, \eqref{equ: z_k+1 proof}, and the choice of $\varepsilon_k$ in \eqref{equ: sigma_k and epsilon_k} that
	\begin{equation}\label{ineq:x-epsilonk}
		\begin{aligned}
&\left\|\proj_{\mathrm{T}_{x_{k+1}}\mathcal{M}}\left(\nabla f(x_{k+1}) +\nabla \mathcal{A}(x_{k+1})^\top z_{k+1} \right) \right\|\\
=&\left\|\proj_{\mathrm{T}_{x_{k+1}}\mathcal{M}}\left(\nabla\mathcal{L}_k(x_{k+1})\right)\right\|	=\left\|\rgrad\mathcal{L}_k(x_{k+1})\right\|\overset{\text{(a)}}{\leq}\varepsilon_k \overset{(b)}{=}\frac{\varepsilon_1}{b^{k-1}},
	\end{aligned}
	\end{equation}
	where (a) follows from \eqref{ineq: subproblem stop} and $x_{k+1} = x_{k, t_k}$ and (b) is due to the update of $\varepsilon_k$ in \eqref{equ: sigma_k and epsilon_k}.
From \eqref{equ:new} and the update of $z_{k+1}$ in \eqref{equ: update z}, we have the following bound on $\|\mathcal{A}(x_{k+1}) - y_{k+1}\|$:   
\begin{equation}\label{ineq:feas epsilon}
	\begin{aligned}
		\|\mathcal{A}(x_{k+1}) - y_{k+1}\| & 
		=  \frac{\|z_{k+1} - z_k\|}{\sigma_k}  
		\leq \frac{\|z_{k+1} \|+\|z_k\|}{\sigma_k}
		  \leq \frac{2L_h}{\sigma_k}=\frac{2L_h}{\sigma_1 b^{k - 1}},
	\end{aligned}
\end{equation}
        where the last equality is due to the update of $\sigma_k$ in \eqref{equ: sigma_k and epsilon_k}.
Combining \eqref{ineq:x-epsilonk}, \eqref{eq:y-epsilonk}, and  \eqref{ineq:feas epsilon}, and the definition of $K$ in \eqref{equ: K Option I}, we conclude that $x_{K+1}$ is an $\varepsilon$-stationary point of problem \eqref{prob:Riemannian nonsmooth composite}.

Next, for the inner iteration complexity, by Lemma \ref{lem: inner complexity}, the RGD method requires at most $\left\lceil 2L_k\Delta\varepsilon_k^{-2}\right\rceil$ iterations to return a point $x_{k, t_k}$ that satisfies the stopping condition \eqref{ineq: subproblem stop}. 
Combining the inner and outer iteration complexity, the total number of RGD updates can be bounded by
\begin{equation*}
	\begin{aligned}
		 \sum_{k = 1}^K \left(\frac{2 L_k\Delta}{\varepsilon_k^2}+1 \right)
		={}&K+ 2\Delta\sum_{k = 1}^K \frac{c_1 \sigma_k + c_2}{\varepsilon_k^2} \\
		={}&K+ \frac{2\Delta}{\varepsilon_1^2b^2}\left(\frac{c_1\sigma_1}{b}\sum_{k = 1}^Kb^{3k}+c_2\sum_{k = 1}^Kb^{2k}\right)\\
		\leq{}& K+\frac{2\Delta}{\varepsilon_1^2b^2}\left(\frac{c_1\sigma_1 b^2}{b^{3}-1}b^{3K}+\frac{c_2 b^2}{b^2-1}b^{2K}\right)\\
		\overset{\text{(a)}}{\leq{}}&K+\frac{2\Delta(c_1 \sigma_1 + c_2)}{\varepsilon_1^2(b^2 -1 )} b^{3K}\\
        \overset{\text{(b)}}{=}{}& \mathcal{O}(\varepsilon^{-3}),
	\end{aligned}
\end{equation*}
where (a) uses $b > 1$ and (b) follows from \eqref{equ: K Option I}. 
This completes the proof.
\end{proof}

Theorem \ref{Th: I} shows that the first-order oracle complexity of Algorithm \ref{Algorithm RiAL} for finding an $\varepsilon$-stationary point of problem \eqref{prob:Riemannian nonsmooth composite} is $\mathcal{O}(\varepsilon^{-3})$, which matches the best-known complexity results for problem \eqref{prob:Riemannian nonsmooth composite} in \cite{deng2024oracle,beck2023dynamic}. 
It is worth noting that while the dual update in RiAL follows a classical approach, the  ManIAL method in \cite{deng2024oracle} computes $z_{k+1}$ using a damped stepsize as follows:
\begin{equation}\label{equ:damped dual stepsize}
	z_{k+1} = z_k+\beta_0\min\left(\frac{\|\mathcal{A}(x_1 )- y_1\|\log^22}{\|\mathcal{A}(x_{k+1})-y_{k+1}\|(k+1)^2\log(k+2)},1\right)(\mathcal{A}(x_{k+1})-y_{k+1}).
\end{equation}
Here, $\beta_0>0$ is a preset constant. Our result demonstrates that this damped dual stepsize in \eqref{equ:damped dual stepsize}, as well as the  additional projection onto a compact set required in the MIAL method \cite{deng2023manifold} in order to bound the dual variable, seems unnecessary.
Moreover, our results work for a general nonlinear mapping $\mathcal{A}$, whereas the results in \cite{beck2023dynamic,deng2024oracle} apply only to the case where $\mathcal{A}$ is a linear mapping. Finally, our analysis requires that the level set $\Omega_{x_1}$ is bounded (i.e., Assumption \ref{assumption: compact level set}), which is much weaker than the boundedness assumptions on the manifold $\mathcal{M}$ made in \cite{deng2023manifold,deng2024oracle}. 
\section{Numerical Results}\label{sec: numerical results} 
In this section, we report the numerical results of RiAL-RGD and ManIAL \cite{deng2024oracle} for solving sparse PCA and sparse CCA problems. 
All tests are implemented in MATLAB 2023b and evaluated on Apple M2 Pro CPU. In all tests, both algorithms are terminated when they return an $\varepsilon$-stationary point or hit the maximum outer iteration number 100. The maximum number of inner iterations for RGD when solving each subproblem is set to 5000. We set $\varepsilon=10^{-5},\varepsilon_1 = \sigma_1= b=1.5$.
To enhance the performance of the inner RGD method, we utilize the Riemannian Barzilai-Borwein stepsize \cite{barzilai1988two,wen2013feasible,jiang2015framework,iannazzo2018riemannian,jiang2022riemannian} as the initial stepsize and perform a backtracking line search to find a suitable stepsize. 
The key difference between RiAL-RGD and ManIAL lies in the stepsize used to update the dual variable; see \eqref{equ: update z} and \eqref{equ:damped dual stepsize}. For ManIAL, the damped dual stepsize in \eqref{equ:damped dual stepsize} with $\beta_0=1$ is used. 
\subsection{Results on Sparse PCA}
Given a data matrix $A\in\mathbb{R}^{d\times N}$, where each of the $N$ columns corresponds to a data sample with $d$ attributes, the sparse PCA problem can be mathematically formulated as \cite{zou2018selective}
\begin{equation}\label{prob: SPCA}
	\min_{X\in\mathcal{S}(d,r)} \left\{\Phi(X):=-\langle AA^\top,\,XX^\top\rangle+\mu\|X\|_1\right\}.
\end{equation}
Here, 
$ \mathcal{S}(d,r) = \{X \in \mathbb{R}^{d \times r}\mid X^\top X = I_r\} $ is the Stiefel manifold with $I_r$ being the $r$-by-$r$ identity matrix, $\mu>0$ is the weighting parameter, and $\|X\|_1=\sum_{i,j}|X_{ij}|$ is the $\ell_1$-norm of the matrix $X$. 
In our tests, we randomly generate the data matrix $A$ as described in \cite{zhou2023semismooth}. 
\begin{table}
	\centering
	\fontsize{8pt}{\baselineskip}\selectfont
	\caption{Average performance comparison on Sparse PCA.}
	\tabcolsep=0.15cm
	\renewcommand\arraystretch{1.2}
	\begin{tabular}{cccccc|ccccc}
		\hline
		 & \multicolumn{5}{c}{ManIAL \cite{deng2024oracle}} & \multicolumn{5}{c}{RiAL-RGD}  \\
		\cline{2-11}
		 & $-\Phi$ &  spar & cpu & outer & total & $-\Phi$ & spar & cpu & outer & total  \\
		\hline
		$\mu$ & \multicolumn{10}{c}{ $d=500,\, N=50,\,r=10$}\\
		\hline
		0.5 & $410.3$ & $31.2$ & $5.9$ & $34$ & $31658$ & $410.4$ &  $31.6$ & $\mathbf{2.9}$ & $\mathbf{22}$ & $\mathbf{18725}$  \\ 
		0.75 & $377.0$ & $39.1$ & $7.0$ & $36$ & $37236$ & $377.2$ &  $39.3$ & $\mathbf{3.4}$ & $\mathbf{24}$ & $\mathbf{22286}$  \\ 
		1 & $346.3$ & $45.5$ & $13.3$ & $43$ & $81793$ & $346.5$ & $45.7$ & $\mathbf{4.3}$ & $\mathbf{24}$ & $\mathbf{27344}$  \\ 
		1.25 & $318.0$ & $50.8$ & $8.9$ & $37$ & $50867$ & $318.5$ & $51.2$ & $\mathbf{3.5}$ & $\mathbf{24}$ & $\mathbf{22572}$  \\ 
		1.5 & $292.2$ & $55.1$ & $9.2$ & $37$ & $49510$ & $292.3$ & $55.4$ & $\mathbf{3.3}$ & $\mathbf{25}$ & $\mathbf{21563}$  \\ 
		\hline
		$r$ & \multicolumn{10}{c}{ $d=1000,\,\, N=50,\,\mu=5$}\\
		\hline        
4 & $327.2$  & $41.6$ & $13.2$ & $40$ & $53399$& $327.5$ & $41.7$ & $\mathbf{3.5}$ & $\mathbf{28}$ & $\mathbf{20950}$  \\ 
5 & $341.5$ & $56.6$ & $20.1$ & $40$ & $53001$ & $342.2$ & $56.7$ & $\mathbf{3.5}$ & $\mathbf{27}$ & $\mathbf{14891}$  \\ 
6 & $324.7$ & $58.8$ & $22.2$ & $40$ & $52699$ & $325.0$ & $58.9$ & $\mathbf{5.8}$ & $\mathbf{28}$ & $\mathbf{23586}$  \\ 
7 & $362.0$ & $67.4$ & $24.3$ & $40$ & $50866$ & $364.2$ & $67.8$ & $\mathbf{5.9}$ & $\mathbf{27}$ & $\mathbf{20328}$  \\ 
8 & $315.9$ & $70.9$ & $26.2$ & $40$ & $52356$ & $316.2$ & $71.1$ & $\mathbf{5.0}$ & $\mathbf{28}$ & $\mathbf{17899}$  \\ 
	\hline
	\end{tabular}
	\label{Tab:  SPCA  aver10}
\end{table}

The average results over 20 runs  with different randomly generated data matrices and initial points are presented in Table \ref{Tab:  SPCA aver10}. In this table, $\Phi$ is the objective value of problem \eqref{prob: SPCA}, “cpu" represents the cpu time in seconds, “outer" denotes the number of outer iterations, and “total" denotes the total number of Riemannian gradient descent steps. 
We also compare the the sparsity of $X$ (denoted by “spar''),
which is defined as the percentage of entries with the absolute value less than $10^{-5}$. 
From Table \ref{Tab:  SPCA  aver10}, we observe that compared with ManIAL, RiAL-RGD can generally find higher-quality solutions in terms of both the value of $\Phi$ and the sparsity. Moreover, RiAL is significantly faster than ManIAL, requiring much fewer outer and total iterations. The higher efficiency of RiAL-RGD mainly benefits from its use of the  classical full stepsize for the dual update, as opposed to the damped stepsize employed by ManIAL. 
\subsection{Sparse CCA}
Given two data matrices $A\in\mathbb{R}^{d\times p}$ and $B\in\mathbb{R}^{d\times q}$, let $\hat{\Sigma}_{aa}=\frac{1}{d} A^{\top}A$ and $ \hat{\Sigma}_{bb}=\frac{1}{d} B^{\top}B$ be the sample covariance matrices of $X$ and $Y$, respectively, and $\hat{\Sigma}_{ab}=\frac{1}{d} A^{\top}B$ be the sample cross-covariance matrix. The sparse CCA can be formulated as \cite{deng2024oracle,chen2019alternating}
\begin{equation}\label{pro:scca}
	\min_{U\in \mathcal{S}_{\Sigma_{aa}}(p,r), V\in \mathcal{S}_{\Sigma_{bb}}(p,r)} \left\{\Phi(U,V):=-\mathrm{tr}(U^\top \hat{\Sigma}_{ab} V) + \mu_1 \|U\|_1 + \mu_2 \|V\|_1\right\}. 
\end{equation}
Here, $\mathcal{S}_{G}(p,r)=\{X\in\mathbb{R}^{p\times r}\mid X^\top GX=I_r\}$ with $G\in\mathbb{R}^{p\times p}$ being a positive definite matrix is the generalized Stiefel manifold and $\mu_1, \mu_2 > 0$ are the weighting parameters. In our tests, the data matrices are randomly generated as in \cite{deng2024oracle} and we set $\mu_1=\mu_2=\mu$.
\begin{table}
	\centering
	\fontsize{8pt}{\baselineskip}\selectfont
	\caption{Average performance comparison on Sparse CCA.}
	\tabcolsep=0.1cm
	\renewcommand\arraystretch{1.2}
	\begin{tabular}{ccccccc|cccccc}
		\hline
		& \multicolumn{6}{c}{ManIAL \cite{deng2024oracle}} & \multicolumn{6}{c}{RiAL-RGD}  \\
		\cline{2-13}
		& $-\Phi$ & sparu & sparv & cpu & outer & total & $-\Phi$ & sparu & sparv & cpu & outer & total  \\
		\hline
		$\mu$ & \multicolumn{12}{c}{$r=5,\,d=1000,\,p=q=200$}\\
		\hline
0.05 & $3.011$ & $32.7$ & $33.0$ & $30.7$ & $28$ & $53253$ & $3.013$ & $32.7$ & $33.7$ & $\mathbf{14.7}$ & $\mathbf{18}$ & $\mathbf{26941}$  \\ 
0.07 & $2.448$ & $42.6$ & $42.2$ & $49.4$ & $36$ & $98864$ & $2.459$ & $44.0$ & $43.8$ & $\mathbf{14.7}$ & $\mathbf{18}$ & $\mathbf{27098}$  \\ 
0.10 & $1.700$ & $51.1$ & $50.9$ & $72.7$ & $51$ & $169284$ & $1.745$ & $58.9$ & $58.5$ & $\mathbf{14.6}$ & $\mathbf{19}$ & $\mathbf{27566}$  \\ 
0.12 & $1.313$ & $54.4$ & $54.8$ & $101.2$ & $66$ & $246947$ & $1.363$ & $67.1$ & $67.1$ & $\mathbf{14.0}$ & $\mathbf{20}$ & $\mathbf{26600}$  \\ 
0.15 & $0.822$ & $66.5$ & $66.7$ & $81.4$ & $59$ & $200863$ & $0.826$ & $78.7$ & $78.4$ & $\mathbf{20.2}$ & $\mathbf{21}$ & $\mathbf{37009}$  \\ 
		\hline
		$r$ & \multicolumn{12}{c}{$\mu=0.05,\,d=1000,\,p=q=200$}\\
		\hline        
		2 & $1.252$ & $30.7$ & $31.1$ & $15.2$ & $28$ & $44210$ & $1.256$ & $31.0$ & $31.1$ & $\mathbf{3.5}$ & $\mathbf{18}$ & $\mathbf{10904}$  \\ 
3 & $1.816$ & $31.1$ & $31.2$ & $49.9$ & $40$ & $114272$ & $1.821$ & $31.9$ & $31.8$ & $\mathbf{9.2}$ & $\mathbf{18}$ & $\mathbf{18601}$  \\ 
4 & $2.410$ & $32.7$ & $32.2$ & $22.2$ & $28$ & $55946$ & $2.418$ & $32.9$ & $32.6$ & $\mathbf{7.2}$ & $\mathbf{18}$ & $\mathbf{19530}$  \\ 
5 & $3.011$ & $32.7$ & $33.0$ & $31.4$ & $28$ & $53253$ & $3.013$ & $32.7$ & $33.7$ & $\mathbf{15.0}$ & $\mathbf{18}$ & $\mathbf{26941}$  \\ 
6 & $3.600$ & $32.4$ & $32.9$ & $52.4$ & $36$ & $100133$ & $3.611$ & $33.2$ & $33.6$ & $\mathbf{16.3}$ & $\mathbf{17}$ & $\mathbf{28998}$  \\
		\hline
	\end{tabular}
	\label{Tab:  SCCA  aver10}
\end{table}

The average results over 20 runs  with different randomly generated data matrices and initial points are presented in Table \ref{Tab:  SCCA aver10}, where “sparu” and “sparv” denote the sparsity of $U$ and $V$, respectively. We can observe from Table \ref{Tab:  SCCA  aver10} that RiAL-RGD always return better solutions in terms of both the value of $\Phi$ in \eqref{pro:scca} and the sparsity of $U$ and $V$. 
Moreover, RiAL-RGD is much more efficient than ManIAL, requiring less CPU time and fewer outer and total iterations. These results again validate the necessity of using the classical full dual stepsize in Riemannian AL methods.

\section{Concluding Remarks}
In this paper, we established the first-order oracle complexity of a Riemannian AL method called RiAL-RGD which utilizes the classical dual update for solving the Riemannian nonsmooth composite problem in \eqref{prob:Riemannian nonsmooth composite}. We proved that RiAL-RGD can find an $\varepsilon$-stationary point of the considered problem with at most $\mathcal{O}(\varepsilon^{-3})$ calls to the first-order oracle, thereby achieving the best-known oracle complexity. Numerical results on sparse PCA and sparse CCA validate the superiority of RiAL compared to an existing Riemannian AL method in \cite{deng2024oracle}. 
~\\[6pt]
\textbf{Declarations}\\[4pt]
\begin{small}
\noindent \textbf{Conflict of interest} 
The authors declare that they have no conflict of interest to this work. \\[4pt]
\noindent \textbf{Data availability} The datasets are generated randomly, with details and citations provided in the corresponding sections.
\end{small}

\bibliographystyle{plain}
\bibliography{reference_arv}






\end{document}